\documentclass[12pt,reqno]{amsart}

\usepackage{amssymb}
\usepackage{amsthm}
\usepackage{amsmath}
\usepackage{a4wide}
\usepackage{latexsym}
\usepackage{mathrsfs}
\usepackage[v2,tips]{xy}
\usepackage{float}
\usepackage[T1]{fontenc}
\usepackage[latin1]{inputenc}

\newtheorem{theorem}{Theorem}[section]
\newtheorem{lemma}[theorem]{Lemma}

\newtheorem{corollary}[theorem]{Corollary}

\theoremstyle{definition}

\newtheorem{remark}[theorem]{Remark}
\newtheorem{question}[theorem]{Question}

\numberwithin{equation}{section}

\newcounter{smallromans}

\newenvironment{romanenumerate}
{\begin{list}{{\normalfont\textrm{(\roman{smallromans})}}}%
    {\usecounter{smallromans}\setlength{\itemindent}{0cm}%
      \setlength{\leftmargin}{5.5ex}\setlength{\labelwidth}{5.5ex}%
      \setlength{\topsep}{0.2ex}\setlength{\partopsep}{0ex}%
      \setlength{\itemsep}{0.2ex}}}%
  {\end{list}}

\newcommand{\romanref}[1]{{\normalfont\textrm{(\ref{#1})}}}

\newcounter{smallalphs}

\newenvironment{alphenumerate}
{\begin{list}{{\normalfont\textrm{(\alph{smallalphs})}}}%
    {\usecounter{smallalphs}\setlength{\itemindent}{0cm}%
      \setlength{\leftmargin}{5.5ex}\setlength{\labelwidth}{5.5ex}%
      \setlength{\topsep}{0.2ex}\setlength{\partopsep}{0ex}%
      \setlength{\itemsep}{0.2ex}}}%
  {\end{list}}

\newcommand{\alphref}[1]{{\normalfont\textrm{(\ref{#1})}}}

\renewcommand{\le}{\ensuremath{\leqslant}}

\renewcommand{\ge}{\ensuremath{\geqslant}}

\newcommand{\N}{\mathbb{N}}

\newcommand{\R}{\mathbb{R}}
\newcommand{\C}{\mathbb{C}}

\renewcommand{\phi}{\ensuremath{\varphi}}
\renewcommand{\epsilon}{\ensuremath{\varepsilon}}

\newcommand{\spa}{\operatorname{span}}

\newcommand{\smashw}[2][l]{{\text{\makebox[0pt][#1]{$#2$}}}}

\begin{document}
\title[Operators on two Banach spaces of continuous functions]{Operators
  on two Banach spaces of continuous func\-tions on locally compact
  spaces of ordinals} \subjclass[2010]%
{Primary 46H10,
47B38, 
47L10;
Secondary 06F30,
46B26,
47L20}
\author[T.~Kania]{Tomasz Kania}
\address{Department of Mathematics and Statistics, Fylde
  College, Lancaster University, Lancaster LA1 4YF, United Kingdom.\newline \indent
  Current address: Institute of Mathematics, Polish Academy
  of Sciences, ul. \'Sniadeckich 8, 00-956 War\-sza\-wa, Poland}
\email{tomasz.marcin.kania@gmail.com}
\author[N.~J.~Laustsen]{Niels Jakob  Laustsen} 
\address{Department of Mathematics and Statistics, Fylde
  College, Lancaster University, Lancaster LA1 4YF, United Kingdom}
\email{n.laustsen@lancaster.ac.uk}

\begin{abstract} 
  Denote by~$[0,\omega_1)$ the set of countable ordinals, equipped
  with the order topo\-logy, let $L_0$ be the disjoint union of the
  compact ordinal intervals~$[0,\alpha]$ for $\alpha$ countable, and
  consider the Banach spaces~$C_0[0,\omega_1)$ and~$C_0(L_0)$
  consisting of all scalar-valued, continuous functions which are
  defined on the locally compact Hausdorff spaces~$[0,\omega_1)$
  and~$L_0$, respectively, and which vanish eventually. Our main
  result states that a bounded, linear operator~$T$ between any pair
  of these two Banach spaces fixes an isomorphic copy of~$C_0(L_0)$ if
  and only if the identity operator on~$C_0(L_0)$ factors through~$T$,
  if and only if the Szlenk index of~$T$ is uncountable. This implies
  that the set $\mathscr{S}_{C_0(L_0)}(C_0(L_0))$ of
  $C_0(L_0)$-strictly singular operators on~$C_0(L_0)$ is the unique
  maximal ideal of the Banach algebra~$\mathscr{B}(C_0(L_0))$ of all
  bounded, linear operators on~$C_0(L_0)$, and that
  $\mathscr{S}_{C_0(L_0)}(C_0[0,\omega_1))$ is the second-largest
  proper ideal of~$\mathscr{B}(C_0[0,\omega_1))$. Moreover, it follows
  that the Banach space~$C_0(L_0)$ is primary and complementably
  homogeneous.
\end{abstract}
\maketitle
\section{Introduction and statement of main results}%
\label{section1}
\noindent 
The purpose of this paper is to advance our understanding of the
(bounded, linear) opera\-tors acting on the Banach
space~$C_0[0,\omega_1)$ of scalar-valued, continuous functions which
vanish eventually and which are defined on the locally compact
Hausdorff space~$[0,\omega_1)$ of all countable ordinals, equipped
with the order topology.  Another Banach space of a similar kind,
$C_0(L_0)$, will play a key role; here~$L_0$ denotes the locally
compact Hausdorff space given by the disjoint union of the compact
ordinal intervals~$[0,\alpha]$ for $\alpha$ countable or,
equi\-valent\-ly, $L_0 = \bigcup_{\alpha<\omega_1} [0,
\alpha]\times\{\alpha+1\}$, endowed with the topology inherited from
the product topology on~$[0,\omega_1)^2$.  As a spin-off of our main
line of inquiry, we obtain some conclusions of possible independent
interest concerning the Banach space~$C_0(L_0)$ and the operators
acting on it.

The motivation behind this work comes primarily from our previous
studies \cite{KL,KKL} (the latter written jointly with P.~Koszmider)
of the Banach algebra~$\mathscr{B}(C_0[0,\omega_1))$ of all operators
on~$C_0[0,\omega_1)$.  Indeed, the main theo\-rem of~\cite{KL} implies
that~$\mathscr{B}(C_0[0,\omega_1))$ has a unique maximal
ideal~$\mathscr{M}$ (the \emph{Loy--Willis ideal}),
while~\cite[Theorem~1.6]{KKL} 
characterizes this ideal as the set of operators which factor
through the Banach space~$C_0(L_0)$:
\begin{equation}\label{MfactoringthroughC0L0}
  \mathscr{M} = \{ TS : S\in\mathscr{B}(C_0[0,\omega_1),C_0(L_0)),\,
  T\in\mathscr{B}(C_0(L_0),C_0[0,\omega_1))\}, 
\end{equation}
using the alternative description of~$C_0(L_0)$ given in
Lemma~\ref{factsaboutEomega1C0omega1}\romanref{EAcongEomega1i}, below.

To state the main result of the present paper precisely, we require
the following three pieces of terminology concerning an operator~$T$
between the Banach spaces~$X$ and~$Y\colon$
\begin{itemize}
\item Let~$W$ be a Banach space. Then~$T$ \emph{fixes an isomorphic
    copy} of~$W$ if there exists an operator $U\colon W\to X$ such
  that the composite operator~$TU$ is an isomorphism onto its range.
\item An operator~$V$ between the Banach spaces~$W$ and~$Z$
  \emph{factors through}~$T$ if there exist operators $R\colon W\to X$
  and $S\colon Y\to Z$ such that $V = STR$.
\item Suppose that $X$ is an Asplund space. Then each weakly$^*$
  compact subset~$K$ of the dual space~$X^*$ of~$X$ has associated
  with it a certain ordinal, which is called the \emph{Szlenk index}
  of~$K$; we refer to \cite{szlenk} or \cite[Section 2.4]{hmvz} for
  the precise definition of this notion. The \emph{Szlenk index} of
  the operator~$T$
  is then defined as the Szlenk index of the image under the adjoint
  operator of~$T$ of the closed unit ball of~$Y^*$.
\end{itemize}

\begin{theorem}\label{mainC0L0thm}
  Let~$X$ and~$Y$ each denote either of the Banach
  spaces~$C_0[0,\omega_1)$ and~$C_0(L_0),$ and let~$T$ be an operator
  from~$X$ to~$Y$.  Then the following three conditions are
  equivalent:
  \begin{alphenumerate}
  \item\label{mainC0L0thm2} $T$ fixes an isomorphic copy of~$C_0(L_0);$
  \item\label{mainC0L0thm1} the identity operator on~$C_0(L_0)$
    factors through~$T;$
  \item\label{mainC0L0thm3} the Szlenk index of~$T$ is uncountable.
  \end{alphenumerate}
\end{theorem}

In the case where $X= Y\in\{C_0[0,\omega_1),C_0(L_0)\}$,
Theorem~\ref{mainC0L0thm} leads to significant new information about
the lattice of closed ideals of the Banach algebra~$\mathscr{B}(X)$ of
all operators on~$X$. (Throughout this paper, the term \emph{ideal}
means a two-sided, algebraic ideal, while a \emph{maximal ideal} is a
proper ideal which is maximal with respect to inclusion among all
proper ideals.)  To facilitate the statement of these conclusions, let
us introduce the notation~$\mathscr{S}_W(X)$ for the set of those
operators on~$X$ that do not fix an isomorphic copy of the Banach
space~$W$; such operators are called \emph{$W$-strictly singular}.
The set~$\mathscr{S}_W(X)$ is closed in the norm topology, and it is
closed under composition by arbitrary operators, so
that~$\mathscr{S}_W(X)$ is a closed ideal of~$\mathscr{B}(X)$ if and
only if~$\mathscr{S}_W(X)$ is closed under addition.

\begin{corollary}\label{theC0L0singularideal} Let $X =
  C_0[0,\omega_1)$ or $X = C_0(L_0)$. Then
  \begin{align}
    \mathscr{S}_{C_0(L_0)}(X) &= \{ T\in\mathscr{B}(X) : \text{the
      identity operator on}\ C_0(L_0)\ \text{does not factor
      through}\ T\}\notag\\ &= \{ T\in\mathscr{B}(X) : \text{the
      Szlenk index of}\ T\ \text{is
      countable}\},\label{theC0L0singularidealEq1}
  \end{align} and this set is a proper closed ideal
  of~$\mathscr{B}(X)$.

  In the case where $X = C_0(L_0),$ this ideal is the unique maximal
  ideal of~$\mathscr{B}(C_0(L_0)),$ whereas for $X = C_0[0,\omega_1),$
  this ideal is the second-largest proper ideal
  of~$\mathscr{B}(C_0[0,\omega_1)),$ in the following precise sense:
  \begin{itemize}
  \item $\mathscr{S}_{C_0(L_0)}(C_0[0,\omega_1))$ is properly
    contained in the Loy--Willis ideal~$\mathscr{M};$ and
  \item for each proper ideal~$\mathscr{I}$
    of~$\mathscr{B}(C_0[0,\omega_1)),$ either $\mathscr{I} =
    \mathscr{M}$ or $\mathscr{I}\subseteq
    \mathscr{S}_{C_0(L_0)}(C_0[0,\omega_1))$.
  \end{itemize}
\end{corollary}

Finally, as an easy Banach-space theoretic consequence of these
results, we shall show that~$C_0(L_0)$ has the following two
properties, defined for any Banach space~$X$:
\begin{itemize}
\item $X$ is \emph{primary} if, for each
  projection~$P\in\mathscr{B}(X)$, either the kernel of~$P$ or the
  range of~$P$ (or both) is isomorphic to~$X$;
\item $X$ is \emph{complementably homogeneous} if, for each closed
  subspace~$W$ of~$X$ such that~$W$ is isomorphic to~$X$, there exists
  a closed, complemented subspace~$Y$ of~$X$ such that~$Y$ is
  isomorphic to~$X$ and~$Y$ is contained in~$W$.
\end{itemize}
\begin{corollary}\label{C0L0primary} The Banach space $C_0(L_0)$ is
  primary and complementably homogeneous. 
\end{corollary}
The counterpart of this corollary is true for~$C_0[0,\omega_1)$; this
is due to Alspach and Benya\-mi\-ni~\cite[Theorem~1]{AB} and the
present authors and Koszmider~\cite[Corollary~1.12]{KKL},
respectively.

\section{Preliminaries}
\noindent
In this section we shall explain our conventions and terminology in
further detail, state some important theorems that will be required in
the proof of Theorem~\ref{mainC0L0thm}, and establish some auxiliary
results.

All Banach spaces are supposed to be over the same scalar field, which
is either the real field~$\R$ or the complex field~$\C$. By an
\emph{operator}, we understand a bounded and linear mapping between
Banach spaces.  We write~$I_X$ for the identity operator on the
Banach~space~$X$.

The following elementary characterization of the operators that the
identity operator factors through is well known.
\begin{lemma}\label{operatorsfactoringID}
  Let $X,$ $Y,$ and~$Z$ be Banach spaces. Then the identity operator
  on~$Z$ factors through an operator $T\colon X\to Y$ if and only if
  $X$ contains a closed subspace~$W$ such that:
  \begin{itemize}
  \item $W$ is isomorphic to~$Z;$
  \item the restriction of~$T$ to~$W$ is bounded below, in the sense
    that there exists a constant $\epsilon>0$ such that
    $\|Tw\|\ge\epsilon\|w\|$ for each $w\in W;$
  \item the image of~$W$ under~$T$ is complemented in~$Y$.
  \end{itemize}
\end{lemma}

For an ordinal~$\alpha$ and a pair~$(X,Y)$ of Banach spaces,
let~$\mathscr{S\!Z}_\alpha(X,Y)$ denote the set of operators $T\colon
X\to Y$ such that the Szlenk index of~$T$ is defined and does not
exceed~$\omega^\alpha$.  Brooker \cite[Theorem~2.2]{brookerJOT} has
proved the following result.

\begin{theorem}[Brooker]\label{brookerthm}
  The class~$\mathscr{S\!Z}_\alpha$ is a closed, injective, and
  surjective operator ideal in the sense of Pietsch for each
  ordinal~$\alpha$.
\end{theorem}

For a Hausdorff space~$K$, we denote by~$C(K)$ the vector space of all
scalar-valued, con\-ti\-nuous functions defined on~$K$. In the case
where~$K$ is locally compact, the subspace~$C_0(K)$ consisting of
those functions $f\in C(K)$ for which the set $\{k\in K :
|f(k)|\ge\epsilon\}$ is compact for each $\epsilon>0$ is a Banach
space with respect to the supremum norm, and we have $C(K) = C_0(K)$
if and only if~$K$ is compact.

The case where $K$ is an ordinal interval (always equipped with the
order topology) will be particularly important for us. Bessaga and
Pe\l{}czy\'{n}ski~\cite{bessagapelczynski} have shown that the Banach
spaces~$C[0,\omega^{\omega^\alpha}]$, where~$\alpha$ is a countable
ordinal, exhaust all possible isomorphism classes of Banach spaces of
the form~$C(K)$ for a countably infinite, compact metric space~$K$.
The Banach spaces $C[0,\omega^{\omega^\alpha}]$ may be viewed as
higher-ordinal analogues of the Banach space~$c_0$, which corresponds
to the case where $\alpha=0$.

The Szlenk index can be used to distinguish these Banach spaces
because Samuel~\cite{samuel} has shown
that~$C[0,\omega^{\omega^\alpha}]$ has Szlenk
index~$\omega^{\alpha+1}$ for each countable ordinal~$\alpha$; a
simplified proof of this result, due to H\'{a}jek and
Lancien~\cite{hl}, is given in \cite[Theorem~2.59]{hmvz}.

More importantly for our purposes, Bourgain~\cite{bourgain} has proved
that each operator~$T$, defined on a $C(K)$-space and of sufficiently
large Szlenk index, fixes an isomorphic copy of~$C[0,\alpha]$ for some
countable ordinal~$\alpha$, which increases with the Szlenk index
of~$T$. The precise statement of this result is as follows.

\begin{theorem}[Bourgain]\label{bourgainszlenk}
  Let $X$ be a Banach space, let~$K$ be a compact Hausdorff space, and
  let~$\alpha$ be a countable ordinal. Then each operator $T\colon
  C(K)\to X$ whose Szlenk index exceeds~$\omega^\alpha$ fixes an
  isomorphic copy of~$C[0,\omega^{\omega\cdot \alpha}]$. \end{theorem}

\begin{remark}
\begin{romanenumerate}
\item Bourgain first proves Theorem~\ref{bourgainszlenk} in the case
  where~$K$ is a compact metric space \cite[Proposition~3]{bourgain},
  and he then explains how to deduce the result for general~$K$
  \cite[p.~107]{bourgain}.
\item Alspach~\cite{alspach} has shown that the most obvious
  strengthening of Bourgain's theorem is false by constructing a
  surjective operator~$T$ on~$C[0,\omega^{\omega^2}]$ such that~$T$
  does not fix an isomorphic copy of~$C[0,\omega^{\omega^2}]$. (The
  surjectivity of~$T$ implies that~$T$ has the same Szlenk index as
  its codomain, that is,~$\omega^3$.)
\end{romanenumerate}
\end{remark}

We shall use Bourgain's theorem in tandem with the following theorem
of Pe\l{}czy\'{n}ski~\cite[Theorem~1]{pelczynski3}, which will enable
us to infer that the identity operator on~$C(K)$, where~$K$ is a
compact metric space, factors through each operator which fixes an
isomorphic copy of~$C(K)$ and which has separable codomain.

\begin{theorem}[Pe\l{}czy\'{n}ski]%
  \label{pelczynskicomplementationthem} Let~$W$ be a closed subspace
  of a separable Banach space~$X,$ and suppose that~$W$ is isomorphic
  to~$C(K)$ for some compact metric space~$K$.  Then~$W$ contains a
  closed subspace which is isomorphic to~$C(K)$ and which is
  complemented in~$X$.
\end{theorem}

To conclude this section, we shall state some results about the Banach
spaces~$C_0[0,\omega_1)$ and~$C_0(L_0)$ that will be required in the
proof of Theorem~\ref{mainC0L0thm}.  As in~\cite{KKL}, it turns out to
be convenient to work with an alternative representation of the Banach
space~$C_0(L_0)$, stated in
Lemma~\ref{factsaboutEomega1C0omega1}\romanref{EAcongEomega1i}, below.
This relies on the following piece of notation. Denote by
\[ \Bigl(\bigoplus_{j\in J} X_j\Bigr)_{c_0} = \bigl\{ (x_j)_{j\in J} :
x_j\in X_j\ (j\in J)\ \text{and}\ \{j\in J : \| x_j\|\ge\epsilon\}\
\text{is finite}\ (\epsilon>0)\bigr\}
\]
the $c_0$-direct sum of a family $(X_j)_{j\in J}$ of Banach spaces,
and set $E_{\omega_1} = \bigl(\bigoplus_{\alpha<\omega_1} C[0,\alpha]
\bigr)_{c_0}$ and, more generally, $E_A = \bigl(\bigoplus_{\alpha\in
  A} C[0,\alpha]\bigr)_{c_0}$ for each non-empty subset~$A$
of~$[0,\omega_1)$.

\begin{lemma}\label{factsaboutEomega1C0omega1}
  \begin{romanenumerate}
  \item\label{EAcongEomega1i} The Banach spaces $C_0(L_0)$ and
    $E_{\omega_1}$ are isometrically isomorphic.
  \item\label{factsaboutC0omega1iii} The Banach space~$E_{\omega_1}$ is
    isomorphic to the $c_0$-direct sum of countably many copies of
    itself: $E_{\omega_1}\cong c_0(\N,E_{\omega_1})$.
  \item\label{EAcongEomega1iii} Let~$A$ be an uncountable subset
    of~$[0,\omega_1)$. Then~$E_A$ is isomorphic to~$ E_{\omega_1}$.
  \item\label{factsaboutC0omega1i} The Banach space~$C_0[0,\omega_1)$
    contains a closed, complemented subspace which is isomorphic
    to~$E_{\omega_1}$.
  \item\label{factsaboutC0omega1ii} The following three conditions are
    equivalent for each operator~$T$ on~$C_0[0,\omega_1)\colon$
    \begin{alphenumerate}
    \item\label{factsaboutC0omega1iia} $T$ belongs to the Loy--Willis
      ideal~$\mathscr{M};$
    \item\label{factsaboutC0omega1iib} the range of~$T$ is contained
      in a closed, complemented subspace which is isomorphic
      to~$E_{\omega_1};$
    \item\label{factsaboutC0omega1iic} the identity operator
      on~$C_0[0,\omega_1)$ does not factor through~$T$.
    \end{alphenumerate}
  \end{romanenumerate}
\end{lemma} 
\begin{proof} Clause~\romanref{EAcongEomega1i} is a special instance
  of a well-known elementary fact (see, \emph{e.g.}, \cite[p.~191,
  Exercise 9]{conway}), while clauses~\romanref{factsaboutC0omega1iii}
  and~\romanref{factsaboutC0omega1i} are \cite[Lemma~2.12 and
  Corollary~2.16]{KKL}, respectively.

  To prove clause~\romanref{EAcongEomega1iii}, we observe that the
  Banach spaces~$E_{\omega_1}$ and~$E_A$ contain complemented co\-pies
  of each other.  By~\romanref{factsaboutC0omega1iii}, the
  Pe\l{}czy\'{n}ski decomposition method (as stated in
  \cite[Theorem~2.23(b)]{ak}, for instance) applies, and hence the
  conclusion follows.

  \romanref{factsaboutC0omega1ii}. The equivalence of
  conditions~\alphref{factsaboutC0omega1iia}
  and~\alphref{factsaboutC0omega1iic} is \cite[Theorem~1.2]{KL} (or
  \cite[Theorem~1.6, (a)$\Leftrightarrow$(h)]{KKL}), while the
  equivalence of conditions~\alphref{factsaboutC0omega1iia}
  and~\alphref{factsaboutC0omega1iib} follows from
  \cite[Theorem~1.6]{KKL}; more precisely,
  \alphref{factsaboutC0omega1iib}
  implies~\alphref{factsaboutC0omega1iia} by
  \cite[Theorem~1.6(d)$\Rightarrow$(a)]{KKL}, and the converse can be
  shown as in the proof of \cite[Theorem~1.6(c)$\Rightarrow$(d)]{KKL}.
\end{proof}

For each countable ordinal~$\alpha$, let~$P_\alpha$ be the norm-one
projection on~$C_0[0,\omega_1)$ given by $P_\alpha f = f\cdot
\mathbf{1}_{[0,\alpha]}$ for each $f\in C_0[0,\omega_1)$,
where~$\mathbf{1}_{[0,\alpha]}$ denotes the characteristic function of
the ordinal interval~$[0,\alpha]$.

\begin{lemma}\label{seprangefromEomega1}
  \begin{romanenumerate}
  \item\label{seprangefromEomega1i} A subspace of~$C_0[0,\omega_1)$ is
    separable if and only if it is contained in the range of the
    projection~$P_\alpha$ for some countable ordinal~$\alpha$.
  \item\label{seprangefromEomega1ii} Let $T$ be an operator
    from~$C_0[0,{\omega_1})$ into a Banach space~$X$. Then $T$ has
    separable range if and only if $T = TP_\alpha$ for some countable
    ordinal~$\alpha$.
\end{romanenumerate}
\end{lemma}
\begin{proof}
  Clause~\romanref{seprangefromEomega1i} is a special case of
  \cite[Lemma~4.2]{KL}.

  \romanref{seprangefromEomega1ii}.  The implication $\Leftarrow$ is
  clear because~$P_\alpha$ has separable range.

  We shall prove the converse by contradiction. Assume that~$T$ has
  separable range and that $T\neq TP_\alpha$ for each
  $\alpha<\omega_1$. Since each element of~$C_0[0,{\omega_1})$ has
  countable support, we may inductively construct a transfinite
  sequence of disjointly supported functions
  $(f_\alpha)_{\alpha<\omega_1}$ in~$C_0[0,{\omega_1})$ such that
  $Tf_\alpha\ne 0$ and $\|f_\alpha\| = 1$ for each
  $\alpha<\omega_1$. Set \[ A(n) = \Bigl\{ \alpha\in[0,\omega_1) : \|
  Tf_\alpha\|\ge\frac{1}{n} \Bigr\}\qquad (n\in\N). \] Since
  $[0,\omega_1) = \bigcup_{n\in\N} A(n)$, we can find $n_0\in\N$ such
  that~$A(n_0)$ is uncountable.

  Take a sequence $(x_m)_{m\in\N}$ which is dense in the range of~$T$,
  and set
  \[ B(m) = \Bigl\{ \alpha\in A(n_0) : \| x_m - Tf_\alpha\|\le
  \frac{1}{3n_0} \Bigr\}\qquad (m\in\N). \] Then, as $A(n_0) =
  \bigcup_{m\in\N} B(m)$, we deduce that~$B(m_0)$ is uncountable for
  some $m_0\in\N$.  Note that $\| x_{m_0}\|\ge2/(3n_0)$.  Now choose
  an integer~$k$ such that $k>3n_0\|T\|$, and take~$k$ distinct
  ordinals $\alpha_1,\ldots,\alpha_k\in B(m_0)$. Since the function
  $\sum_{j=1}^k f_{\alpha_j}\in C_0[0,\omega_1)$ has norm one, we
  conclude that
  \[ \|T\|\ge \biggl\|\sum_{j=1}^k Tf_{\alpha_j} \biggr\|\ge
  k\|x_{m_0}\| - \sum_{j=1}^k \|x_{m_0} -
  Tf_{\alpha_j}\|\ge\frac{k}{3n_0}> \|T\|, \] which is clearly absurd.
\end{proof}

For each countable ordinal~$\alpha$, let~$Q_\alpha$ be the canonical
projection of~$E_{\omega_1}$ onto the first~$\alpha$ summands; that
is, the $\beta^{\text{th}}$ coordinate of
$Q_\alpha(f_\gamma)_{\gamma<\omega_1}$, where $\beta<\omega_1$ and
\mbox{$(f_\gamma)_{\gamma<\omega_1}\in E_{\omega_1}$}, is given
by~$f_\beta$ if $\beta\le\alpha$ and~$0$ otherwise. We then have the
following counterpart of
Lemma~\ref{seprangefromEomega1}\romanref{seprangefromEomega1i}.

\begin{lemma}\label{sepsubspaceofEomega}
  A subspace of~$E_{\omega_1}$ is separable if and only if it is
  contained in the range of the projection~$Q_\alpha$ for some
  countable ordinal~$\alpha$.
\end{lemma}
\begin{proof}
  The implication $\Leftarrow$ is clear.

  Conversely, suppose that $W$ is a separable subspace
  of~$E_{\omega_1}$, and let~$D$ be a countable, dense subset
  of~$W$. Since each element of~$E_{\omega_1}$ has countable support,
  for each $x\in D$, we can choose a countable ordinal~$\beta(x)$ such
  that $x = Q_{\beta(x)}x$. Then the ordinal $\alpha = \sup_{x\in
    D}\beta(x)$ is countable and satisfies $x = Q_\alpha x$ for each
  $x\in D$, and hence $W\subseteq Q_\alpha[E_{\omega_1}]$.
\end{proof}

\section{Proofs of Theorem~\ref{mainC0L0thm} and
  Corollaries~\ref{theC0L0singularideal}--\ref{C0L0primary}}
\noindent
We are now ready to prove the results stated in
Section~\ref{section1}. We begin with a lemma which, in the light of
Lemma~\ref{factsaboutEomega1C0omega1}\romanref{EAcongEomega1i}, above,
effectively establishes Theorem~\ref{mainC0L0thm} in the case where $X
= C_0[0,\omega_1)$ and $Y = C_0(L_0)$.

\begin{lemma}\label{inductivesteplemma}
  Let~$T\colon C_0[0,\omega_1)\to E_{\omega_1}$ be an operator with
  uncountable Szlenk index. Then:
  \begin{romanenumerate}
  \item\label{inductivesteplemma1} for each pair $(\alpha,\eta)$ of
    countable ordinals, there exist operators $R\colon C[0,\alpha]\to
    C_0[0,\omega_1)$ and \mbox{$S\colon E_{\omega_1}\to C[0,\alpha]$}
    and a countable ordinal $\xi > \eta$ such that the diagram
    \[ \spreaddiagramrows{4ex}\spreaddiagramcolumns{-4.5ex}%
    \xymatrix{%
      & C[0,\alpha]\ar^-{\displaystyle{I_{C[0,\alpha]}}}[rrrrrrrr]%
      \ar_-{\displaystyle{R}}[ld]%
      &&&&&&&& C[0,\alpha]\\
      C_0[0,\omega_1)\ar_-{\displaystyle{P_\xi - P_\eta}}[rd]
      &&&&&&&&&&
      \quad\ E_{\omega_1}\quad\ \ar_-{\displaystyle{S}}[lu]\\
      & C_0[0,\omega_1) \ar^-{\displaystyle{T}}[rrrrrrrr] &&&&&&&&
      E_{\omega_1} \ar_-{\displaystyle{Q_\xi - Q_\eta}}[ru]} \] is
    commutative;
  \item\label{inductivesteplemma2} the identity operator
    on~$E_{\omega_1}$ factors through~$T$.
  \end{romanenumerate}
\end{lemma}

\begin{proof} \romanref{inductivesteplemma1}. Let $U =
  (I_{E_{\omega_1}} - Q_\eta)T(I_{C_0[0,\omega_1)} - P_\eta)$. Since
  $T = U + Q_\eta T + TP_\eta - Q_\eta TP_\eta$, where each of the
  final three terms~$Q_\eta T$, $TP_\eta$, and~$Q_\eta TP_\eta$ has
  countable Szlenk index, Theorem~\ref{brookerthm} implies that~$U$
  has uncountable Szlenk index. Hence~$U$ fixes an isomorphic copy
  of~$C[0,\alpha]$ by Theorem~\ref{bourgainszlenk}; that is, we can
  find a closed subspace~$W$ of~$C_0[0,\omega_1)$ such that~$W$ is
  isomorphic to~$C[0,\alpha]$ and the restriction of~$U$ to~$W$ is
  bounded below.
  Lemmas~\ref{seprangefromEomega1}\romanref{seprangefromEomega1i}
  and~\ref{sepsubspaceofEomega} enable us to choose a countable
  ordinal~$\xi>\eta$ such that the separable subspaces~$W$ and~$U[W]$
  are contained in the ranges of the projections~$P_\xi$ and~$Q_\xi$,
  respectively. Then the re\-stric\-tions to~$W$ of the operators~$U$
  and $(Q_\xi-Q_\eta)T(P_\xi - P_\eta)$ are equal. By
  The\-orem~\ref{pelczynskicomplementationthem}, $U[W]$~contains a
  closed subspace~$Z$ which is isomorphic to~$C[0,\alpha]$ and which
  is complemented in~$Q_\xi[E_{\omega_1}]$, and therefore also
  in~$E_{\omega_1}$. The conclusion now follows from
  Lemma~\ref{operatorsfactoringID} because~$U$, and thus
  $(Q_\xi-Q_\eta)T(P_\xi - P_\eta)$, maps the closed subspace $W\cap
  U^{-1}[Z]$ isomorphically onto~$Z$.

  \romanref{inductivesteplemma2}.
  Using~\romanref{inductivesteplemma1} together with
  Lemmas~\ref{seprangefromEomega1}\romanref{seprangefromEomega1ii}
  and~\ref{sepsubspaceofEomega}, we may inductively construct
  trans\-finite se\-quen\-ces of countable ordinals
  $(\eta(\alpha))_{\alpha<\omega_1}$ and
  $(\xi(\alpha))_{\alpha<\omega_1}$ and of ope\-ra\-tors
  \mbox{$(R_\alpha\colon C[0,\alpha]\to
    C_0[0,{\omega_1}))_{\alpha<\omega_1}$} and $(S_\alpha\colon
    E_{\omega_1}\to C[0,\alpha])_{\alpha<\omega_1}$ such that
    \mbox{$\eta(\alpha) < \xi(\alpha) < \eta(\beta)$},
  \begin{equation}\label{mainC0L0thmEq2}
    I_{C[0,\alpha]} = S_\alpha(Q_{\xi(\alpha)} - 
    Q_{\eta(\alpha)})T(P_{\xi(\alpha)} - P_{\eta(\alpha)})R_\alpha,
  \end{equation} 
  \begin{equation}\label{mainC0L0thmEq1}
    (I_{E_{\omega_1}} - Q_{\eta(\beta)})TP_{\xi(\alpha)} = 0,\qquad 
    \text{and}\qquad Q_{\xi(\alpha)}T(I_{C_0[0,\omega_1)} - P_{\eta(\beta)}) = 0
  \end{equation} 
  whenever $0\le\alpha<\beta<\omega_1$.  We may clearly also suppose
  that $\|R_\alpha\| = 1$ for each $\alpha<\omega_1$.  

  Given $n\in\N$, set $A(n) = \{\alpha\in[0,\omega_1) : \|
  S_\alpha\|\le n\}$.  Since $[0,\omega_1) = \bigcup_{n\in\N} A(n)$,
  we conclude that $A(n_0)$ is uncountable for some
  $n_0\in\N$. Then~$E_{A(n_0)}$ is isomorphic to~$E_{\omega_1}$ by
  Lemma~\ref{factsaboutEomega1C0omega1}\romanref{EAcongEomega1iii}, so
  that it will suffice to show that the identity operator
  on~$E_{A(n_0)}$ factors through~$T$.

  To this end, we observe that
  \[ S\colon\ x\mapsto \bigl(S_\alpha(Q_{\xi(\alpha)} -
  Q_{\eta(\alpha)})x\bigr)_{\alpha\in A(n_0)},\quad E_{\omega_1}\to
  E_{A(n_0)}, \] defines an operator of norm at most~$n_0$. Moreover,
  introducing the subspaces \[ F_\beta = \bigl\{ (f_\alpha)_{\alpha\in
    A(n_0)}\in E_{A(n_0)} : f_\alpha = 0\
  (\alpha\ne\beta)\bigr\}\qquad (\beta\in A(n_0)), \] we can define a
  linear contraction by
  \[ R\colon\ (f_\alpha)_{\alpha\in A(n_0)}\mapsto \sum_{\alpha\in
    A(n_0)}(P_{\xi(\alpha)} - P_{\eta(\alpha)})R_\alpha f_\alpha,\quad
  \spa\bigcup_{\beta\in A(n_0)}F_\beta\to C_0[0,{\omega_1}). \] Since
    the domain of definition of~$R$ is dense in~$E_{A(n_0)}$,
    $R$~extends uniquely to a linear contraction defined
    on~$E_{A(n_0)}$. Now, given $\beta\in A(n_0)$ and
    $(f_\alpha)_{\alpha\in A(n_0)}\in F_\beta$, we have
  \[ STR(f_\alpha)_{\alpha\in A(n_0)} = \bigl(S_\alpha(Q_{\xi(\alpha)}
  - Q_{\eta(\alpha)})T(P_{\xi(\beta)} - P_{\eta(\beta)})R_\beta
  f_\beta\bigr)_{\alpha\in A(n_0)} = (f_\alpha)_{\alpha\in A(n_0)} \]
  by~\eqref{mainC0L0thmEq2} and the fact that $(Q_{\xi(\alpha)} -
  Q_{\eta(\alpha)})T(P_{\xi(\beta)} - P_{\eta(\beta)}) = 0$ for
  $\alpha\ne\beta$ by~\eqref{mainC0L0thmEq1}.  This implies that $STR
  = I_{E_{A(n_0)}}$, and the result follows.
\end{proof}

\begin{proof}[Proof of Theorem~{\normalfont{\ref{mainC0L0thm}}}.]
  The implications
  \alphref{mainC0L0thm1}$\Rightarrow$\alphref{mainC0L0thm2}%
  $\Rightarrow$\alphref{mainC0L0thm3} are clear, so it remains to
  prove that
  \alphref{mainC0L0thm3}$\Rightarrow$\alphref{mainC0L0thm1}. Hence, we
  suppose that the Szlenk index of~$T$ is uncountable, and seek to
  demonstrate that the identity operator on~$C_0(L_0)$ factors
  through~$T$. Throughout the proof, we shall tacitly use the fact
  that~$C_0(L_0)$ is isomorphic to~$E_{\omega_1}$, \emph{cf.}\
  Lemma~\ref{factsaboutEomega1C0omega1}\romanref{EAcongEomega1i}.
  
  Lemma~\ref{inductivesteplemma}\romanref{inductivesteplemma2}
  establishes the desired conclusion for $X = C_0[0,\omega_1)$ and $Y
  = C_0(L_0)$.

  Now suppose that $X = Y = C_0[0,\omega_1)$. If $T\notin\mathscr{M}$,
  then the identity operator on~$C_0[0,\omega_1)$ factors through~$T$
  by Lemma~\ref{factsaboutEomega1C0omega1}%
  \romanref{factsaboutC0omega1ii}, and hence the identity operator
  on~$C_0(L_0)$ also factors through~$T$ by
  Lemma~\ref{factsaboutEomega1C0omega1}\romanref{factsaboutC0omega1i}.
  Otherwise $T\in\mathscr{M}$, in which case
  condition~\alphref{factsaboutC0omega1iib} of
  Lemma~\ref{factsaboutEomega1C0omega1}\romanref{factsaboutC0omega1ii}
  implies that $T = VUT$ for some operators $U\colon
  C_0[0,\omega_1)\to C_0(L_0)$ and $V\colon C_0(L_0)\to
  C_0[0,\omega_1)$. Then~$UT$ has the same Szlenk index as~$T$, so
  that, by the first case, the identity operator on~$C_0(L_0)$ factors
  through~$UT$, and hence through~$T$.

  Finally, suppose that $X = C_0(L_0)$, and either $Y =
  C_0[0,\omega_1)$ or $Y = C_0(L_0)$.  By
  Lemma~\ref{factsaboutEomega1C0omega1}\romanref{factsaboutC0omega1i},
  we can take operators \mbox{$U\colon C_0(L_0)\to C_0[0,\omega_1)$}
  and \mbox{$V\colon C_0[0,\omega_1)\to C_0(L_0)$} such that
  $I_{C_0(L_0)} = VU$. Then~$TV$ has the same Szlenk index as~$T$, and
  the conclusion follows from the previous two cases, as above.
\end{proof}

\begin{proof}[Proof of
  Corollary~{\normalfont{\ref{theC0L0singularideal}}}.]
  The three sets in~\eqref{theC0L0singularidealEq1} are equal by the
  negation of Theorem~\ref{mainC0L0thm}. The final of these sets is
  clearly equal to
  $\bigcup_{\alpha<\omega_1}\mathscr{S\!Z}_\alpha(X)$, which is an
  ideal of~$\mathscr{B}(X)$ by Theorem~\ref{brookerthm}.
  
  For $X = C_0(L_0)$, the fact that the second set
  in~\eqref{theC0L0singularidealEq1} is an ideal of
  $\mathscr{B}(C_0(L_0))$ implies that it is necessarily the unique
  maximal ideal by an observation of Dosev and
  Johnson (see~\cite[p.~166]{dj}).

  For $X = C_0[0,\omega_1)$, we have
  $\mathscr{S}_{C_0(L_0)}(C_0[0,\omega_1))\subsetneqq\mathscr{M}$
  because $\mathscr{S}_{C_0(L_0)}(C_0[0,\omega_1))$ is an ideal
  of~$\mathscr{B}(C_0[0,\omega_1))$, $\mathscr{M}$ is the unique
  maximal ideal, and any projection on~$C_0[0,\omega_1)$ whose range
  is isomorphic to~$C_0(L_0)$ is an example of an operator which
  belongs to
  \mbox{$\mathscr{M}\setminus\mathscr{S}_{C_0(L_0)}(C_0[0,\omega_1))$}.
    To verify the final statement of the corollary, suppose
  that~$\mathscr{I}$ is a proper ideal of
  $\mathscr{B}(C_0[0,\omega_1))$ such that~$\mathscr{I}$ is not
  contained in~$\mathscr{S}_{C_0(L_0)}(C_0[0,\omega_1))$.  Then,
  by~\eqref{theC0L0singularidealEq1}, $\mathscr{I}$ contains an
  operator~$T$ such that $I_{C_0(L_0)} = STR$ for some operators
  $R\colon C_0(L_0)\to C_0[0,\omega_1)$ and \mbox{$S\colon
    C_0[0,\omega_1)\to C_0(L_0)$}. Given $U\in\mathscr{M}$, we can
  find operators $V\colon C_0[0,\omega_1)\to C_0(L_0)$ and $W\colon
  C_0(L_0)\to C_0[0,\omega_1)$ such that $U = WV$
  by~\eqref{MfactoringthroughC0L0}, and hence $U =
  (WS)T(RV)\in\mathscr{I}$. This proves that $\mathscr{M}\subseteq
  \mathscr{I}$, and consequently $\mathscr{M} = \mathscr{I}$.
\end{proof}

\begin{proof}[Proof of Corollary~{\normalfont{\ref{C0L0primary}}}.]
  First, to show that~$C_0(L_0)$ is primary,
  let~$P\in\mathscr{B}(C_0(L_0))$ be a projection.  Since the
  ideal~$\mathscr{S}_{C_0(L_0)}(C_0(L_0))$ is proper, it cannot
  contain both~$P$ and \mbox{$I_{C_0(L_0)} - P$}; we may without loss
  of generality suppose that
  \mbox{$P\notin\mathscr{S}_{C_0(L_0)}(C_0(L_0))$}.  Then,
  by~\eqref{theC0L0singularidealEq1}, $I_{C_0(L_0)}$~factors
  through~$P$, so that $P[C_0(L_0)]$ contains a closed, complemented
  subspace which is isomorphic to~$C_0(L_0)$ by
  Lemma~\ref{operatorsfactoringID}.  Since~$P[C_0(L_0)]$ is
  complemented in~$C_0(L_0)$, and the Pe\l{}czy\'{n}ski decomposition
  method (as stated in \cite[Theorem~2.23(b)]{ak}, for instance)
  applies by Lemma~\ref{factsaboutEomega1C0omega1}%
  \romanref{EAcongEomega1i}--\romanref{factsaboutC0omega1iii}, we
  conclude that~$P[C_0(L_0)]$ and~$C_0(L_0)$ are isomorphic,
  as desired.

  Second, to verify that~$C_0(L_0)$ is complementably homogeneous,
  suppose that~$W$ is a closed subspace of~$C_0(L_0)$ such that~$W$ is
  isomorphic to~$C_0(L_0)$. Take an isomorphism~$U$ of~$C_0(L_0)$
  onto~$W$, and let~$J$ be the natural inclusion of~$W$
  into~$C_0(L_0)$. Since the opera\-tor $JU\in\mathscr{B}(C_0(L_0))$
  fixes an isomorphic copy of~$C_0(L_0)$, we can find operators~$R$
  and~$S$ on~$C_0(L_0)$ such that $I_{C_0(L_0)} = S(JU)R$ by
  Theorem~\ref{mainC0L0thm}. The operator \mbox{$P =
  JURS\in\mathscr{B}(C_0(L_0))$} is then a projection, its range is
  clearly contained in~$W$, and the restriction of~$S$ to the range
  of~$P$ is an isomorphism onto~$C_0(L_0)$ (with inverse~$JUR$).
\end{proof}

\section{Closing remarks}
\noindent 
The Banach algebras~$\mathscr{B}(C_0[0,\omega_1))$
and~$\mathscr{B}(C[0,\omega_1])$ are isomorphic because the
under\-lying Banach spaces~$C_0[0,\omega_1)$ and~$C[0,\omega_1]$ are
isomorphic.  In particular~$\mathscr{B}(C_0[0,\omega_1))$
and~$\mathscr{B}(C[0,\omega_1])$ have isomorphic lattices of closed
ideals. The study of this lattice was initiated in~\cite{KL}, while
Corollary~\ref{theC0L0singularideal} of the present paper adds another
element to our understanding of it. Figure~\ref{diagramClosedIdeals},
below, summarizes our current knowledge of this lattice,
extending~\cite[Figure~1]{KL}.

\begin{figure}[ht]%
  $\spreaddiagramrows{-1ex}%
  \xymatrix{%
    &&&\mathscr{B}\\%
    &&&\smashw[r]{\mathscr{M}}=\smashw[l]{\mathscr{G}_{C_0(L_0)}}%
    \ar@{^{(}->>}[u]\\%
    \mathscr{X}\,\ar@{^{(}->}[r]
    &\mathscr{X}+\overline{\mathscr{G}}_{c_0(\omega_1)}%
    \ar@{^{(}->}[r] &
    \bigcup_{1\le\alpha<\omega_1}\overline{\mathscr{G}}_{c_0(\omega_1,\,
      C(K_{\alpha}))}%
    \ar[r] &\mathscr{S}_{C_0(L_0)} \ar@{^{(}->>}[u]\\%
    \vdots\ar@{^{(}->>}[u] &\vdots%
    \ar@{^{(}->>}[u] &\vdots \ar@{^{(}->>}[u]
    &\vdots\ar@{^{(}->>}[u]\\%
    \overline{\mathscr{G}}_{C(K_{\alpha+1})}%
    \ar@{^{(}->}[r]\ar@{^{(}->}[u]%
    &\overline{\mathscr{G}}_{C(K_{\alpha+1})\oplus c_0(\omega_1)}%
    \ar@{^{(}->}[u]\ar@{^{(}->}[r]
    &\overline{\mathscr{G}}_{c_0(\omega_1,\, C(K_{\alpha+1}))}%
    \ar@{^{(}->}[u]\ar[r] &\mathscr{S\!Z}_{\alpha+2}\ar@{^{(}->}[u]\\%
    \overline{\mathscr{G}}_{C(K_{\alpha})}\ar@{^{(}->}[u]\ar@{^{(}->}[r]
    &\overline{\mathscr{G}}_{C(K_{\alpha})\oplus c_0(\omega_1)}%
    \ar@{^{(}->}[u]\ar@{^{(}->}[r]
    &\overline{\mathscr{G}}_{c_0(\omega_1,\, C(K_\alpha))}%
    \ar[r]\ar@{^{(}->}[u] &\mathscr{S\!Z}_{\alpha+1}\ar@{^{(}->}[u]\\%
    \vdots\ar@{^{(}->}[u] & \vdots\ar@{^{(}->}[u] &
    \vdots\ar@{^{(}->}[u]&\vdots\ar@{^{(}->}[u]\\%
    \overline{\mathscr{G}}_{C(K_1)}\ar@{^{(}->}[u]\ar@{^{(}->}[r]
    &\overline{\mathscr{G}}_{C(K_1)\oplus c_0(\omega_1)}%
    \ar@{^{(}->}[u]\ar@{^{(}->}[r]
    &\overline{\mathscr{G}}_{c_0(\omega_1,\, C(K_1))}%
    \ar[r]\ar@{^{(}->}[u] &\mathscr{S\!Z}_2\ar@{^{(}->}[u]\\%
    \overline{\mathscr{G}}_{c_0}\ar@{^{(}->}[u]\ar@{^{(}->>}[r] &
    \overline{\mathscr{G}}_{c_0(\omega_1)}\ar@{^{(}->}[u]\ar[rr]
    &&\mathscr{S\!Z}_1\ar@{^{(}->}[u]\\%
    \mathscr{K}\ar@{^{(}->>}[u]\\%
    \{0\}\ar@{^{(}->>}[u] }$
  \caption{Partial structure of the lattice of  closed ideals of
    $\mathscr{B} =
    \mathscr{B}(C_0[0,\omega_1))$.}\label{diagramClosedIdeals} 
\end{figure}
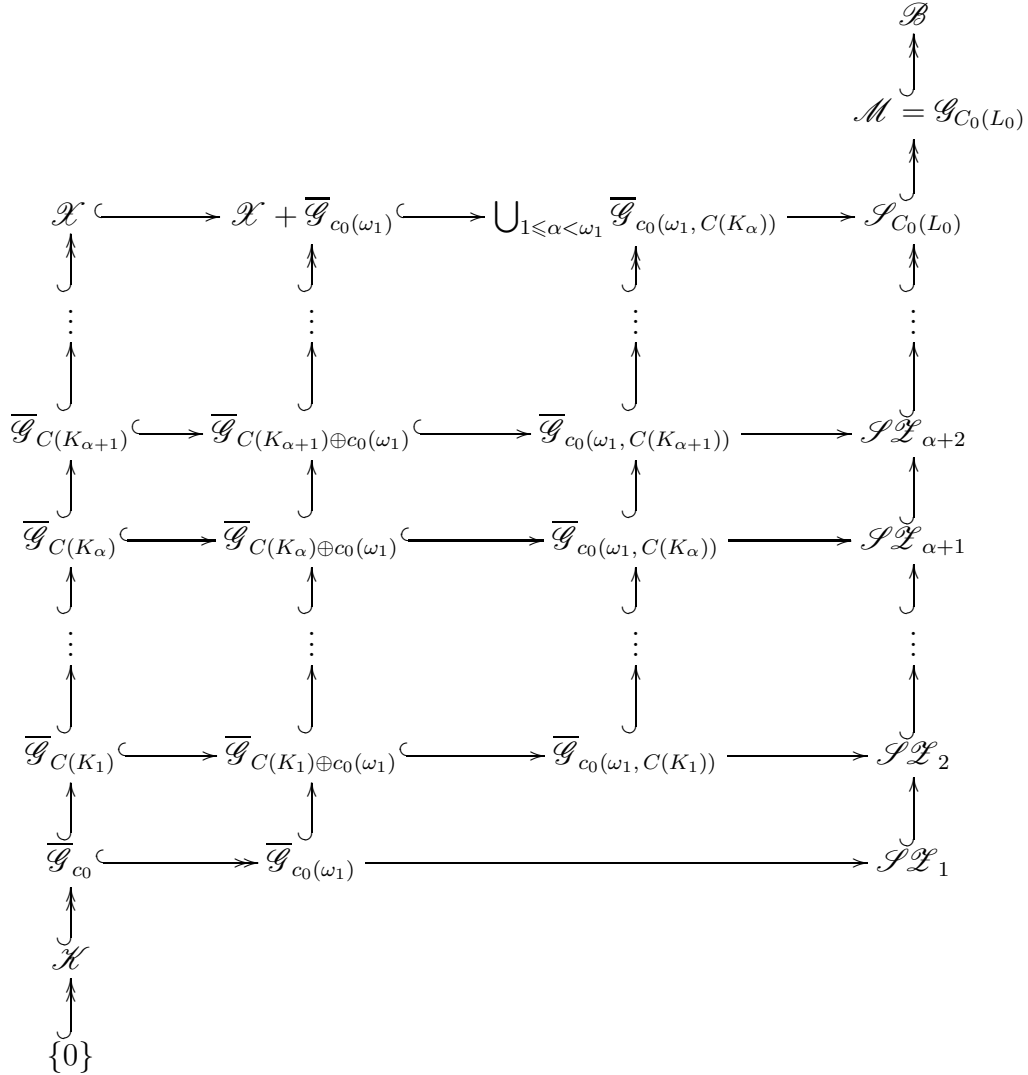

\noindent
In addition to the notation that has already been introduced,
Figure~\ref{diagramClosedIdeals} relies on the following
con\-ven\-tions: the Banach space~$C_0[0,\omega_1)$ is suppressed
everywhere, so that we write~$\mathscr{B}$ in\-stead
of~$\mathscr{B}(C_0[0,\omega_1))$, \emph{etc.}; $\mathscr{K}$~denotes
the ideal of compact operators, $\mathscr{X}$~denotes the ideal of
operators with separable range, $\mathscr{G}_X$~denotes the ideal of
operators that factor through the Banach space~$X$,
and~$\overline{\mathscr{G}}_X$ denotes its closure in the operator
norm; $c_0(\omega_1, X)$~denotes the $c_0$-direct sum of~$\omega_1$
copies of the Banach space~$X$; we write~$c_0(\omega_1)$ in the case
where~$X$ is the sca\-lar field; $\spreaddiagramcolumns{-1.3ex}%
\xymatrix{\mathscr{I}\ar[r]&\!\mathscr{J}}$ signifies that the ideal
$\mathscr{I}$ is contained in the ideal $\mathscr{J}$,
$\spreaddiagramcolumns{-1.3ex}%
\xymatrix{\mathscr{I}\,\ar@{^{(}->}[r]&\!\mathscr{J}}$ in\-di\-cates
that the inclusion is proper, while $\spreaddiagramcolumns{-1.3ex}%
\xymatrix{\mathscr{I}\,\ar@{^{(}->>}[r]&\!\mathscr{J}}$ means that
there are no closed ideals between~$\mathscr{I}$ and~$\mathscr{J}$;
and finally $K_\alpha=[0,\omega^{\omega^\alpha}]$, where~$\alpha$ is a
countable or\-di\-nal.

The relations among the ideals in the first and second columns of
Figure~\ref{diagramClosedIdeals} were all established
in~\cite[Section~5]{KL}. The double-headed arrows at the top of these
columns are meant to signify that
\begin{align}
  \mathscr{X}(C_0[0,\omega_1)) &= \bigcup_{\alpha<\omega_1}
  \overline{\mathscr{G}}_{C(K_\alpha)}(C_0[0,\omega_1))%
  \label{EqTopdiagram1}\\
  \intertext{and} (\mathscr{X} +
  \overline{\mathscr{G}}_{c_0(\omega_1)})(C_0[0,\omega_1)) &=
  \bigcup_{\alpha<\omega_1}\overline{\mathscr{G}}_{C(K_\alpha)\oplus
    c_0(\omega_1)}(C_0[0,\omega_1)),\label{EqTopdiagram2}
\end{align} 
respectively. Indeed, \eqref{EqTopdiagram1} is immediate from
Lemma~\ref{seprangefromEomega1}\romanref{seprangefromEomega1i} (which
shows that it is in fact not necessary to take the closure of the
ideals~$\mathscr{G}_{C(K_\alpha)}(C_0[0,\omega_1))$ on the right-hand
side), and~\eqref{EqTopdiagram2} can easily be verified
using~\eqref{EqTopdiagram1} together with the fact that the ideal
\mbox{$(\mathscr{X} +
  \overline{\mathscr{G}}_{c_0(\omega_1)})(C_0[0,\omega_1))$} is closed
by (the easy first part of) \cite[Theorem~5.8]{KL}.  The counter\-part
of~\eqref{EqTopdiagram1}--\eqref{EqTopdiagram2} for the third column
is obvious, while~\eqref{theC0L0singularidealEq1} establishes the
corresponding identity for the final column.

We shall next show that
\begin{equation}\label{inclusionsrowalpha}
  \overline{\mathscr{G}}_{C(K_{\alpha})\oplus
    c_0(\omega_1)}(C_0[0,\omega_1))\subsetneqq%
  \overline{\mathscr{G}}_{c_0(\omega_1,\, C(K_\alpha))}(C_0[0,\omega_1))\subseteq
  \mathscr{S\!Z}_{\alpha+1}(C_0[0,\omega_1))
\end{equation}
for each countable ordinal $\alpha\ge 1$.  The first inclusion is
immediate because $c_0(\omega_1, C(K_\alpha))$ contains a closed,
complemented subspace which is isomorphic to~$C(K_{\alpha})\oplus
c_0(\omega_1)$, while the second follows from
\cite[Theorem~2.12]{brookerJFA}, which states that (the identity
operator on) the Banach space~$c_0(\omega_1, C(K_\alpha))$ has the
same Szlenk index as~$C(K_\alpha)$, that is,~$\omega^{\alpha+1}$, so
that $I_{c_0(\omega_1,\, C(K_\alpha))}$ belongs
to~$\mathscr{S\!Z}_{\alpha+1}$, and this operator ideal is closed.  To
see that the first inclusion in~\eqref{inclusionsrowalpha} is proper,
we assume the contrary and seek a contradiction. The assumption
implies that there exists an isomorphism~$U$ of~$C(K_{\alpha})\oplus
c_0(\omega_1)$ onto~$c_0(\omega_1, C(K_\alpha))$ (see, \emph{e.g.},
\cite[Lemma~3.1]{BJL}).  Denote by $P\in\mathscr{B}(C(K_\alpha)\oplus
c_0(\omega_1))$ the canonical projection with range~$C(K_\alpha)$ and
kernel~$c_0(\omega_1)$. Then~$P$ belongs to
\mbox{$\mathscr{X}(C(K_\alpha)\oplus c_0(\omega_1))$} and
$I_{C(K_\alpha)\oplus c_0(\omega_1)} - P$ belongs
to~$\mathscr{S\!Z}_1(C(K_\alpha)\oplus c_0(\omega_1))$, so that $Q =
UPU^{-1}$ belongs to~$\mathscr{X}(c_0(\omega_1, C(K_\alpha)))$, while
$I_{c_0(\omega_1,\, C(K_\alpha))} - Q$ belongs
to~$\mathscr{S\!Z}_1(c_0(\omega_1, C(K_\alpha)))$. This, however, is
impossible because, arguing as in the proof of
Lemma~\ref{sepsubspaceofEomega}, we see that since the range of~$Q$ is
separable, it is contained in the first~$\beta$ summands
of~$c_0(\omega_1, C(K_\alpha))$ for some countable ordinal~$\beta$, so
that the $(\beta+1)^{\text{st}}$ summand is contained in the kernel
of~$Q$, and hence~$I_{c_0(\omega_1,\, C(K_\alpha))} - Q$ has Szlenk
index (at least)~$\omega^{\alpha+1}$. This contradiction completes our
proof of~\eqref{inclusionsrowalpha}.  

We do not know whether the second inclusion
in~\eqref{inclusionsrowalpha} is proper. Since this appears to be an
interesting question, we shall raise it formally.
\begin{question}\label{question25dec2013} Are the
  ideals~$\overline{\mathscr{G}}_{c_0(\omega_1,\, C(K_\alpha))}(C_0[0,\omega_1))$
  and~$\mathscr{S\!Z}_{\alpha+1}(C_0[0,\omega_1))$ equal for some, or
  possibly all, countable ordinals~$\alpha\ge1$?
\end{question}
\noindent
It is important that the underlying Banach space is~$C_0[0,\omega_1)$
in this question; indeed, it is easy to find an example of a Banach
space~$X$ on which there is an operator~$T$ which has Szlenk index at
most~$\omega^{\alpha+1}$, but~$T$ does not factor approximately
through~$c_0(\omega_1, C(K_\alpha))$. For instance, take $X = \ell_2$
and $T = I_X$; then the Szlenk index of~$T$ is~$\omega$, but~$T$ does
not factor approximately through~$c_0(\omega_1, C(K_\alpha))$ because
no closed (complemented) subspace of~$c_0(\omega_1, C(K_\alpha))$ is
isomorphic to an infinite-dimensional Hilbert space.

Finally, to see that the inclusions
$\overline{\mathscr{G}}_{c_0(\omega_1,\, C(K_\alpha))}(C_0[0,\omega_1))\subseteq
\overline{\mathscr{G}}_{c_0(\omega_1,\, C(K_{\alpha+1}))}(C_0[0,\omega_1))$
and $\mathscr{S\!Z}_{\alpha+1}(C_0[0,\omega_1))\subseteq
\mathscr{S\!Z}_{\alpha+2}(C_0[0,\omega_1))$ are proper for each
countable ordinal~$\alpha\ge1$, we observe that, for $\beta =
\omega^{\omega^{\alpha+1}}$, the projection~$P_\beta$ defined
immediately before Lemma~\ref{seprangefromEomega1} belongs to
$\mathscr{G}_{C(K_{\alpha+1})}(C_0[0,\omega_1))%
\setminus\mathscr{S\!Z}_{\alpha+1}(C_0[0,\omega_1))$.

One may wonder whether it is possible to obtain a complete description
of the lattice of closed ideals of $\mathscr{B}(C_0[0,\omega_1))$. As
explained in~\cite[p.~4834]{KL}, we consider this an extremely
difficult problem because ``in the first instance, one would need to
classify the closed ideals of
$\mathscr{B}(C[0,\omega^{\omega^\alpha}])$ for each countable
ordinal~$\alpha$, something that already appears intractable; it has
currently been achieved only in the simplest case $\alpha = 0$, where
$C[0,\omega]\cong c_0$.''

\bibliographystyle{amsplain}

\end{document}